\documentclass[12pt,reqno]{amsart}
\usepackage{stmaryrd}
\usepackage{amscd}
\usepackage{pifont}
\usepackage{amsmath}
\usepackage{amsthm}
\usepackage{amsfonts}
\usepackage{amssymb}
\usepackage{latexsym}
\usepackage{amstext}
\usepackage{array}
\usepackage{graphicx}
\usepackage{slashbox}
\usepackage{url}

\date{}
\allowdisplaybreaks[4] \footskip=15pt
\renewcommand{\uppercasenonmath}[1]{}

\newtheorem{thm}[subsection]{Theorem}
\newtheorem{cor}[subsection]{Corollary}
\newtheorem{conj}[subsection]{Conjection}
\newtheorem{Def}[subsection]{Definition}
\newtheorem{lem}[subsection]{Lemma}
\newtheorem{remark}{Remark}

\newtheorem{prop}[subsection]{Proposition}
\newtheorem{exm}[subsection]{Example}
  
\newcommand{\bthm}{\begin{thm} }
\newcommand{\ethm}{\end{thm} }
\newcommand{\bpro}{\begin{prop}}
\newcommand{\epro}{\end{prop}}
\newcommand{\bdf}{\begin{Def}}
\newcommand{\edf}{\end{Def}}
\newcommand{\bexm}{\begin{exm}}
\newcommand{\eexm}{\end{exm}}
\newcommand{\blem}{\begin{lem}}
\newcommand{\elem}{\end{lem}}
\newcommand{\bpf}{\begin{proof}}
\newcommand{\epf}{\end{proof}}
\newcommand{\bcor}{\begin{cor}}
\newcommand{\ecor}{\end{cor}}
\newcommand{\bcon}{\begin{Conj}}
\newcommand{\econ}{\end{Conj}}

\newcommand{\bea}{\begin{eqnarray}}
\newcommand{\eea}{\end{eqnarray}}
\newcommand{\brem}{\begin{remark}}
\newcommand{\erem}{\end{remark}}

\input txdtools

\begin{document}

\title{ \bf Maximal height of divisors of $x^{pq^{b}}-1$}
\keywords{Cyclotomic polynomial; Height of a cyclotomic polynomial}
\author{Shaozu Wang}
\address{Department of Mathematics, Nanjing University, Nanjing 210093, China}
\email{shaozuwang@gmail.com}
\thanks{Supported by NSFC (Nos. 11171141, 11071110),
 NSFJ (Nos. BK2010007, BK2010362), PAPD and the Cultivation Fund of the Key
Scientific and Technical Innovation Project, Ministry of Education
of China (No.708044).}
\maketitle
\begin{abstract}   The height of a polynomial $f(x)$ is the largest coefficient of $f(x)$ in absolute value. Let B(n) be the largest height of a polynomial in $\mathbb{Z}[x]$ dividing $x^n-1$.
In this paper we investigate the maximal height of divisors of $x^{pq^b}-1$ and prove that some conjectures on the maximal height of divisors of $x^{pq^b}-1$ are true.\end{abstract}



\section{\bf Introduction}
 The height H($f(x)$) of a polynomial $f(x)$ is the largest coefficient of $f(x)$ in absolute value.
Let $$\Phi_{n}(x)={\prod_{\substack{1\leq a\leq n \\(a,n)=1}}}(x-e^{\frac{2\pi ia}{n}})$$ be the $n$-th cyclotomic polynomial. The degree
of $\Phi_{n}(x)$ is $\phi(n)$ where $\phi$ is the Euler totient function. We have the factorization
$$x^{n}-1=\prod_{d\mid n}\Phi_{d}(x).$$
All of the polynomials in this paper will be assumed to have integral coefficients.
 The function
$\textrm{A}(n):=\textrm{H}(\Phi_{n}(x))$ was studied extensively (\cite{b1,b2,f1,k2}). Due to the following result, which can be verified directly, to determine A($n$)
it suffices to consider square-free values of $n$.
\begin{prop} Let $p$ be a prime. Then

 \emph{(1) } If $p\mid n$, then $\Phi_{pn}(x)=\Phi_{n}(x^{p})$.

   \emph{(2) } If $p\nmid n$, then $\Phi_{n}(x)\Phi_{pn}(x)=\Phi_{n}(x^{p})$.\end{prop}

This proposition implies that if $p|n$, then $\textrm{A}(pn)=\textrm{A}(n)$. It is also easy to verify that if $n$ is odd, then
$\Phi_{2n}(x)=\Phi_{n}(-x)$. So $\textrm{A}(2n)=\textrm{A}(n)$.

In \cite{p1} Pomerance and Ryan introduced the function $$\textrm{B}(n):={\rm max}\{\textrm{H}(f): f\mid x^{n}-1\ \textrm{and}\ f\in \mathbb{Z}[x]\}.$$
They proved that $\textrm{B}(n)=1$ if and only if $n=p^{l}$
and $\textrm{B}(pq)$=min\{$p,q$\} where $p$ and $q$ are distinct primes. Kaplan [5] obtained that $\textrm{B}(pq^{2})$=min\{$p$,$q^{2}$\}
for any distinct primes $p$ and $q$. And Ryan et al \cite{r1} proved a lower bound for $\textrm{B}(p^aq^b)$ and made the following conjectures.

\begin{conj}\emph{([8])} For a fixed odd prime $p$ and fixed positive integer $b$, the finite list of values $\emph{B}(pq^{b})$ as $p<q$ varies are all divisible by $p$.\end{conj}
\begin{conj}\emph{([8])} Let $p<q$ be odd primes.

\emph{(1)} For any positive integer $b$, $\emph{B}(2q^{b})=2$.

\emph{(2)} For any positive integer $b>2$, $\emph{B}(pq^{b})>p$.\end{conj}
In \cite{d1}, Decker and Moree established various results suggesting that the divisors of $x^{n}-1$ have the tendency to be strongly coefficient convex and have small coefficients. The case where $n=pq^2$ with $p$ and $q$ primes is studied in detail.

We would eventually like to give an explicit formula for $\textrm{B}(pq^{b})$, but this appears
very difficult. In this paper, we will give this classification in the special case where $p<q$ and prove that the above conjectures are ture.
In section 2, we give some lemmas which will play important roles in the proof of our main results. In section 3, we  prove that Conjecture 1.2  as follows.
\bthm Let $p<q$ be primes. For any positive integer $b$, $p|\emph{B}(pq^{b})$. \ethm
By Theorem 1.4, we will show that the first part of Conjecture 1.3 is true. Moreover, we will give an explicit formula for $\textrm{B}(3q^{b})$ as follows.

\bthm Let $p<q$ be primes. For any positive integer $b$, $$\emph{B}(pq^b)=\left\{\begin{array}{ll}
2,&{\rm if}\;\;p=2,\\
3\cdot 2^{[\frac{b-1}{2}]},&{\rm if}\;\;p=3.\end{array}\right.$$  \ethm

In section 4, we will give an explicit formula for $\textrm{B}(pq^3)$ as follows
\bthm Let $p,\;q$ be distinct odd primes and $\rho$ and $\sigma$ denote the unique positive integers such that $\rho p+\sigma q=(p-1)(q-1)$. Then
$$\emph{B}(pq^3)= {\rm max}\{{\rm min}\{p,q^3\},\emph{H}(\Phi_{p}(x)\Phi_{q}(x)\Phi_{pq^2}(x)\Phi_{q^{3}}(x))\},$$ i.e.
$$\emph{B}(pq^3)=\left\{\begin{array}{ll}\emph{H}(\Phi_{p}(x)\Phi_{q}(x)\Phi_{pq^{2}}(x)\Phi_{q^{3}}(x)),&{\rm if}\;\;p<q,\\
{\rm max}\{p,\ \emph{H}(\Phi_{p}(x)\Phi_{q}(x)\Phi_{pq^{2}}(x)\Phi_{q^{3}}(x)) \},&{\rm if}\;\;q< p< q^2,\\
p,&{\rm if}\;\;q^2< p< q^3,\\
q^3,&{\rm if}\;\;p> q^3.\end{array}\right.$$ Moreover, if $p<q$, we have ${\rm B}(pq^3)={\rm max}\{(\sigma+1)p,\ (p-(\sigma+1))p\}$.\ethm
 As an application of Theorem 1.6, the second part of Conjecture 1.3 is true. The rest of the paper consider the formulas for $\textrm{B}(pq^4)$ and $\textrm{B}(pq^5)$. We shall prove the following result
\bthm Let $p<q$ be  distinct odd primes. Then$$\emph{B}(pq^b)=\left\{\begin{array}{ll}{\rm max}\{p\emph{H}(\Phi_{pq}(x)\Phi_{pq^{2}}(x)\Phi_{q^{3}}(x)), \;\; p\emph{H}(\Phi_{pq}(x)\Phi_{q^{2}}(x))\},&{\rm if}\;\;b=4,\\

p\emph{H}^2(\Phi_{pq}(x)\Phi_{q^{2}}(x)),&{\rm if}\;\;b=5.\end{array}\right.$$ \ethm
By Theorem 1.7, we prove that $\textrm{B}(pq^b)$ is determined by
$q(\mbox{mod}\  p)$.
 \bthm Let $p<q<r$ be distinct odd primes and $q\equiv \pm r\ (\mbox{mod}\  p)$. If $b\leq 5$, then $\emph{B}(pq^b)=\emph{B}(pr^b)$.\ethm

\section{\bf Preliminaries }
To show our main results, some preparations are needed.

\begin{lem}\emph{([5])}Let $\emph{T}(f(x))$ denote the sum of the absolute values of the coefficients of the polynomials $f(x)$. Then for any two polynomials $f(x)$ and $g(x)$, $\emph{H}(f(x)g(x))\leq \emph{T}(f(x))\emph{H}(g(x))$ and $\emph{T}(f(x)g(x))\leq \emph{T}(f(x))\emph{T}(g(x))$. \end{lem}

\begin{lem} Let $f(x)=\sum_{i=0}^{n}a_{i}x^{i}$ and $g(x)=\sum_{j=0}^{m}b_{j}x^{kj}$. If there exists $s\in \mathbb{Z}$ such that $sk>n$, then $\emph{H}(f(x)g(x))\leq s\emph{H}(f(x))\emph{H}(g(x))$.\end{lem}
\begin{proof} Since $f(x)g(x)=\sum_{i=0}^{n}\sum_{j=0}^{k}a_{i}b_{j}x^{i+mj}$, the coefficient of $x^{l}$ in the product $f(x)g(x)$ is
$\sum_{i+kj=l}a_{i}b_{j}$ and $\mid a_{i}b_{j}\mid\leq \textrm{H}(f(x))\textrm{H}(g(x))$ where $0\leq l\leq n+mk$. If there exist integers $0\leq i_{1}< i_{2}<\cdot\cdot\cdot< i_{s+1}\leq n$, $0\leq j_{s+1}< j_{s}<\cdot\cdot\cdot< j_{1}\leq m$ such that $$x^{i_{1}+kj_{1}}=x^{i_{2}+kj_{2}}=\cdot\cdot\cdot=x^{i_{s+1}+kj_{s+1}},$$ then $i_{s+1}-i_{1}=k(j_{1}-j_{s+1})$. Hence we have $n\geq sm$ which obviously does not hold. Therefore we have $\mid \sum_{i+mj=l}a_{i}b_{j}\mid\leq s\textrm{H}(f(x))\textrm{H}(g(x))$ and then $$\textrm{H}(f(x)g(x))\leq s\textrm{H}(f(x))\textrm{H}(g(x)).$$\end{proof}\qed

By Lemma 2.2, for distinct primes $p, q$, we have
\begin{eqnarray}
\begin{cases}
\textrm{H}(\Phi_{q}(x)g(x^{q}))=\textrm{H}(g(x^{q}))=\textrm{H}(g(x)),\cr \textrm{H}(\Phi_{q}(x^p)g(x^{q}))=\textrm{H}(g(x^{q}))=\textrm{H}(g(x)), &\textrm{if}\ (p,q)=1.
\end{cases}
\end{eqnarray}
For example, if $b>1$, we have $\textrm{H}(\Phi_{q}(x^p)\Phi_{pq^b}(x))=1$.

We will use the structure of the coefficients of $\Phi_{pq}(x)$ to get some results. Let $\Phi_{pq}(x)=\sum_{i=0}^{(p-1)(q-1)}a_{i}x^{i}$. In [7], Lam and Leung gave a detailed analysis of these coefficients $a_{i}$ and proved the following proposition.
\begin{prop}  Let $p$ and $q$ be distinct primes and $\rho$ and $\sigma$ denoted the unique positive integers such that $\rho p+\sigma q=(p-1)(q-1)$. Then $a_{k}=1$
if and only if $k=ip+jq$ for some $i\in [0,\ \rho]$ and $j\in [0,\ \sigma]$. Also $a_{k}=-1$
if and only if $k+pq=ip+jq$ for some $i\in [\rho+1,\ q-1]$ and $j\in [\sigma+1,\ p-1]$. Otherwise $a_{k}=0$.\end{prop}
\begin{lem} Let $f(x)=\sum_{s=0}^{n}a_{s}x^{qs}$. Then $$\emph{H}(\Phi_{pq}(x)f(x))\leq {\rm max}\{\sigma+1,\ (p-(\sigma+1))\}\emph{H}(f(x)).$$\end{lem}
\begin{proof}
Let
 \begin{eqnarray}
\begin{cases}
h_{1}(x)=\sum_{s=0}^{n}\sum_{i=0}^{\rho}\sum_{j=0}^{\sigma}a_{s}x^{ip+jq+sq},\cr
h_{2}(x)=\sum_{s=0}^{n}\sum_{i=0}^{q-\rho-2}\sum_{j=0}^{p-\sigma-2}a_{s}x^{ip+jq+sq+1}.
\end{cases}
\end{eqnarray}
Then $\Phi_{pq}(x)f(x)=h_{1}(x)-h_{2}(x)$. By Proposition 2.3, $h_{1}(x)$ and $h_{2}(x)$ have no same monomials. Therefore $\textrm{H}(\Phi_{pq}(x)f(x))={\rm max}\{\textrm{H}(h_{1}(x)),\ \textrm{H}(h_{2}(x))\}.$
Fix $(i_{1},j_{1},s_{1})$. If there exists $(i,j,s)$ such that $x^{i_{1}p+j_{1}q+s_{1}q}=x^{ip+jq+sq}$, then $i=i_{1}$. Therefore the coefficient of $x^{i_{1}p+j_{1}q+s_{1}q}$ is $\sum_{j+s=j_{1}+s_{1}}a_{s}$. Since $0\leq j\leq \sigma$ and $|a_{s}|\leq \textrm{H}(f(x))$, we have $|\sum_{j+s=j_{1}+s_{1}}a_{s}|\leq (\sigma+1)\textrm{H}(f(x)).$ Therefore $\textrm{H}(h_{1}(x))\leq (\sigma+1)\textrm{H}(f(x))$.
A similar argument can show that $\textrm{H}(h_{1}(x))\leq (p-(\sigma+1))\textrm{H}(f(x)).$ Then $\textrm{H}(\Phi_{pq}(x)f(x))\leq {\rm max}\{(\sigma+1),\ (p-(\sigma+1))\}\textrm{H}(f(x)).$ This completes the proof of Lemma 2.4.
\end{proof}\qed
\begin{lem}   Let $p$ and $q$ be distinct primes and $g_{0}(x)=\Phi_{pq}(x)\Phi_{q^2}(x)$. Then

 \emph{(1) } If $g(x)\mid \Phi_{q}(x)\Phi_{pq}(x)\Phi_{q^2}(x)\Phi_{pq^2}(x)$ and $g(x)\neq g_{0}(x)$, then $\emph{H}(g(x))=1$.

   \emph{(2) } $1\leq \emph{H}(g_{0}(x))\leq p-1$.\end{lem}

\begin{proof} It's obvious by Lemma 2.1 and Lemma 2.2.
\end{proof}
\qed

\begin{lem} Let $g(x)=\Phi_{pq}(x)\Phi_{pq^2}(x)\Phi_{q^3}(x)=\sum_{i=0}^{deg g(x)}c_{i}x^{i}$. Then we have $c_{i}=c_{i-p}$ if $i\not\equiv 0,\;1\ (\mbox{mod}\  q^2)$.\end{lem}
\begin{proof} We have
\begin{align*}
(x^p-1)g(x)&= \sum_{i>deg g(x)-p}^{deg g(x)}c_{i}x^{i+p}+\sum_{i=0}^{deg g(x)-p}(c_{i}-c_{i+p})x^{i+p}+x-1  \\
&=\sum_{i=0}^{p+q-2}a_{i}(x^{iq^2+1}-x^{iq^2}),
\end{align*}
where
\begin{eqnarray}
a_{i}=
\begin{cases}
i+1, &\textrm{if}\ 0\leq i\leq p-1,\cr p, &\textrm{if}\ p\leq i\leq q,\cr p+q-1-i, & \textrm{if}\ q+1\leq i\leq p+q-2.
\end{cases}
\end{eqnarray}
If $i\not\equiv 0,\;1\ (\mbox{mod}\  q^2)$, the coefficient of $x^i$ in the polynomial $(x^p-1)g(x)$ is $0$, then we have $c_{i}-c_{i-p}=0$. Therefore $c_{i}=c_{i-p}$. Also we can show that $c_{i_{1}}=c_{i_{2}}$ if $rq^2+1<i_{1}<i_{2}<(r+1)q^2$ and $i_{1}\equiv i_{2}(\mbox{mod}\  p)$.\end{proof}
\qed

\begin{lem} \emph{([8])} For primes $p\neq q$, $\emph{B}(p^aq^b)\geq{\rm min}\{p^a, q^b\}$.\end{lem}

\section{\bf Proof of Conjectures 1.2 and 1.3}

 Now we are ready to prove our main results.

\textbf{Proof of Theorem 1.4.} Note that $\textrm{B}(pq)=\textrm{B}(pq^2)=p$, it suffices to consider the case $b>2$. For any $h(x)|(x^{pq^b}-1)$, there exist $f(x)|(x^{pq}-1)$
and $g(x)\mid \prod_{i=1}^{b-1}\Phi_{q^i}(x)\Phi_{pq^i}(x)$ such that $h(x)=f(x)g(x^q)$. We will first prove that $\textrm{H}(f(x)g(x^q))\leq p\textrm{H}(g(x^q)).$ We consider two cases.

{\bf Case 1.} $\textrm{T}(f(x))\leq p$ or ${\rm deg} f(x)<q$. By Lemma 2.1, we have $$\textrm{H}(f(x)g(x^q))\leq p\textrm{H}(g(x^q)).$$

{\bf Case 2.} $\textrm{T}(f(x))>p$ and ${\rm deg} f(x)\geq q$.

(a) $f(x)=\Phi_{pq}(x)$ or $\Phi_{q}(x^p)$ or $\Phi_{p}(x)\Phi_{q}(x)$ or $(x^p-1)\Phi_{q}(x)$ or $\Phi_{1}(x)\Phi_{p}(x^q)$ or $\Phi_{1}(x)\Phi_{q}(x^p)$ or $\Phi_{q}(x)\Phi_{p}(x^q)$.

By Lemma 2.1 and Lemma 2.2, it is easy to show that $$\textrm{H}(f(x)g(x^q))\leq p\textrm{H}(g(x^q)).$$

(b) $f(x)=\Phi_{1}(x)\Phi_{pq}(x)$.

Let $g(x^q)=\sum_{i=0}^{n}a_{qi}x^{qi}$ and \begin{eqnarray}
\begin{cases}
h_{1}(x)=\sum_{s=0}^{n}\sum_{i=0}^{\rho}\sum_{j=0}^{\sigma}a_{qs}x^{ip+jq+sq+1},\cr h_{2}(x)=\sum_{s=0}^{n}\sum_{i=0}^{q-\rho-2}\sum_{j=0}^{p-\sigma-2}a_{qs}x^{ip+jq+sq+1},\cr
h_{3}(x)=\sum_{s=0}^{n}\sum_{i=0}^{q-\rho-2}\sum_{j=0}^{p-\sigma-2}a_{qs}x^{ip+jq+sq+2},\cr
h_{4}(x)=\sum_{s=0}^{n}\sum_{i=0}^{\rho}\sum_{j=0}^{\sigma}a_{qs}x^{ip+jq+sq}.
\end{cases}
\end{eqnarray}
Then $f(x)g(x^q)=h_{1}(x)+h_{2}(x)-h_{3}(x)-h_{4}(x)$. We  consider the height of the polynomial $h_{1}(x)$.
Fix $(i_{1}, j_{1}, s_{1})$. If there exists $(i, j, s)$ such that $$ip+jq+sq=i_{1}p+j_{1}q+s_{1}q,$$ then we have $i=i_{1}$. Therefore the coefficient of $x^{i_{1}p+j_{1}q+s_{1}q}$ is
$\sum_{j+s=j_{1}+s_{1}}a_{qs}$. Since $|a_{qs}|\leq \textrm{H}(g(x^q))$ and $0\leq j\leq \sigma$, we have $$|\sum_{j+s=j_{1}+s_{1}}a_{qs}|\leq (\sigma+1)\textrm{H}(g(x)).$$ Hence $\textrm{H}(h_{1}(x))\leq (\sigma+1)\textrm{H}(g(x^q))$. Similarly, $\textrm{H}(h_{2}(x))\leq (p-\sigma-1)\textrm{H}(g(x^q))$ and then $\textrm{H}(h_{1}(x)+h_{2}(x))\leq p\textrm{H}(g(x^q)).$ We similarly show that $$\textrm{H}(h_{3}(x)+h_{4}(x))\leq p\textrm{H}(g(x^q)).$$By Proposition 2.3, we find that $h_{1}(x)+h_{2}(x)$ and $h_{3}(x)+h_{4}(x)$ have no same monomials.
Therefore $\textrm{H}(f(x)g(x^q))\leq p\textrm{H}(g(x^q)).$

Combining (1) and (2), one obtains $\textrm{H}(f(x)g(x^q))\leq p\textrm{H}(g(x^q))$.

Next we will prove that $\textrm{H}(\Phi_{p}(x)\Phi_{q}(x)g(x^q))=p\textrm{H}(g(x^q))$. Let $g(x^q)=\sum_{i=0}^{n}a_{qi}x^{qi}$ and $a_{qi_{0}}=\textrm{H}(g(x^q))$. The coefficient of $x^{qi_{0}+q-1}$ of $\Phi_{p}(x)\Phi_{q}(x)g(x^q)$ is $pa_{qi_{0}}$,
then we have $\textrm{H}(\Phi_{p}(x)\Phi_{q}(x)g(x^q))\geq p\textrm{H}(g(x^q))$. Therefore $$\textrm{H}(\Phi_{p}(x)\Phi_{q}(x)g(x^q))=p\textrm{H}(g(x^q))=p\textrm{H}(g(x)).$$

Let $\textrm{H}_{b}={\rm max}\{\textrm{H}(g(x))|\       g(x)\mid \prod_{i=1}^{b-1}\Phi_{q^i}(x)\Phi_{pq^i}(x)\}$. We have
$\textrm{B}(pq^{b})=p\textrm{H}_{b}$. This completes the proof of Theorem 1.4.\hfill$\Box$\\

This result shows that given any prime $q>p$, if we want to find the value
of $\textrm{B}(pq^b)$, it suffices to consider the value of $\textrm{H}_{b}$.
\vskip 2mm
\textbf{Proof of Theorem 1.5.}  To prove the theorem, we consider two cases separately.

{\bf Case 1.} $p=2$. Since $\textrm{B}(2q)=\textrm{B}(2q^2)=2$, it is sufficient to consider the case $b>2$.
By Theorem 1.4, $\textrm{B}(2q^b)=2$ is equivalent to $\textrm{H}_{b}=1$. We will first prove that $\textrm{H}_{b}=\textrm{H}_{b-1}$. Obviously, $\textrm{H}_{b}\geq\textrm{H}_{b-1}$. So we need to show $\textrm{H}_{b}\leq \textrm{H}_{b-1}$. For any $h(x)|\prod_{i=1}^{b-1}\Phi_{q^i}(x^2)$, there exist $f(x)|\Phi_{q}(x^2)$ and $g(x)|\prod_{i=2}^{b-1}\Phi_{q^i}(x^2)$ such that $h(x)=f(x)g(x)$. Then $\textrm{H}(h(x))=\textrm{H}(f(x)g(x))$.
If $f(x)=1$ or $\Phi_{q}(x)$ or $\Phi_{q}(x^2)$, by Lemma 2.2, we have $\textrm{H}(f(x)g(x))\leq\textrm{H}(g(x))\leq \textrm{H}_{b-1}.$
If $f(x)=\Phi_{2q}(x)$, then $\Phi_{2q}(x)=\Phi_{q}(-x)$. By Lemma 2.2, we have
$\textrm{H}(\Phi_{2q}(x)g(x))\leq \textrm{H}(g(x))\leq \textrm{H}_{b-1}.$
Then we have $\textrm{H}_{b}\leq \textrm{H}_{b-1}$.
By Lemma 2.5, we have $\textrm{H}_{3}=1$ and then $\textrm{H}_{b}=1$.
Therefore $\textrm{B}(2q^b)=1$.

{\bf Case 2.}  $p=3$. We will first prove that $\textrm{H}_{3}=\textrm{H}_{4}=2$. By Lemma 2.5, we have $\textrm{H}_{3}\leq 2$. Since $q>3$, we have $q\equiv  \pm 1\ (\mbox{mod}\  3)$. Without loss of generality assume that $q\equiv  -1\ (\mbox{mod}\  3)$, then $\sigma=1$ and we have the coefficient of $x^q$ of polynomial $\Phi_{pq}(x)\Phi_{q^{2}}(x)$ is $2$. Hence $\textrm{H}_{3}=2$.
 A similar argument can show that $\textrm{H}_{4}=2$.

Next we will prove that, if $b>4$, then $\textrm{H}_{b}\leq 2\textrm{H}_{b-2}$.
For any $h(x)|\prod_{i=1}^{b-1}\Phi_{q^i}(x^3)$, there exist $f(x)|\prod_{i=1}^{2}\Phi_{q^i}(x^3)$ and $g(x)|\prod_{i=3}^{b-1}\Phi_{q^i}(x^3)$ such that $h(x)=f(x)g(x)$. Then $\textrm{H}(h(x))=\textrm{H}(f(x)g(x))$.
If $f(x)\neq\Phi_{3q}(x)\Phi_{3q^2}(x)$, by Lemmas 2.2 and 2.4, we have
$\textrm{H}(f(x)g(x))\leq 2\textrm{H}(g(x))\leq 2\textrm{H}_{b-2}.$
If $f(x)=\Phi_{3q}(x)\Phi_{3q^2}(x)$, by Lemma 2.2, we have
$$\textrm{H}(f(x)g(x))\leq 2\textrm{H}(\Phi_{3q}(x)\Phi_{3q^2}(x))\textrm{H}(g(x))=2\textrm{H}(g(x))\leq 2\textrm{H}_{b-2}.$$
Therefore, if $b>4$, we have $\textrm{H}_{b}\leq 2\textrm{H}_{b-2}$.

Now we will prove that $\textrm{B}(3q^b)=3\cdot2^{[\frac{b-1}{2}]}$. Since $\textrm{H}_{3}=\textrm{H}_{4}=2$, we have $\textrm{H}_{b}\leq 2^{[\frac{b-1}{2}]}$.
If $b=2k+1$, we
consider the height of polynomial $\prod_{i=1}^{k}\Phi_{3q^{2i-1}}(x)\Phi_{q^{2i}}(x)$. Without loss of generality assume that $q\equiv  -1\ (\mbox{mod}\  3)$, then $\sigma=1$ and we have the coefficient of $x^{q(1+q^2+\cdot\cdot\cdot+q^{2k-2})}$ of polynomial $\prod_{i=1}^{k}\Phi_{pq^{2i-1}}(x)\Phi_{q^{2i}}(x)$ is $2^k$. Therefore we have $\textrm{H}_{2k+1}= 2^{k}$.
If $b=2k+2$, then $2^{k}=\textrm{H}_{2k+1}\leq \textrm{H}_{2k+2}\leq 2^{k}$.
This completes the proof of the case $p=3$.
 \hfill$\Box$\

{\bf Remark: } Let $p<q$ be distinct primes and $q\equiv  \pm 1\ (\mbox{mod}\  p)$.  For any positive integer $b$, we can show that $\textrm{B}(pq^b)=p\cdot(p-1)^{[\frac{b-1}{2}]}$.

\section{ \bf Explicit formula for $\textrm{B}(pq^3)$}

\textbf{Proof of Theorem 1.6.} Since $p,q$ are different primes, we have to deal with two cases.

 {\bf Case 1.}  $p<q$. By Theorem 1.4, it suffices to compute the value of $\textrm{H}_{3}$. So $\textrm{H}_{3}=\textrm{H}(g_{0}(x))$ follows from Lemma 2.5 where $g_{0}(x)=\Phi_{pq}(x)\Phi_{q^{2}}(x)$. Therefore we have $\textrm{B}(pq^3)=p\textrm{H}(g_{0}(x))$.
Now we will prove that $$\textrm{H}(\Phi_{pq}(x)\Phi_{q^{2}}(x))={\rm max}\{\sigma+1,\ p-(\sigma+1)\}.$$ By Proposition 2.3, we have the polynomial
$$\Phi_{pq}(x)\Phi_{q^{2}}(x)=\sum_{i=0}^{\rho}\sum_{j=0}^{\sigma}\sum_{s=0}^{q-1}x^{ip+jq+sq}-\sum_{i=0}^{q-\rho-2}\sum_{j=0}^{p-\sigma-2}\sum_{s=0}^{q-1}x^{ip+jq+sq+1}.$$ If there exist integers $0\leq i_{1},i_{2}\leq \rho$, $0\leq j_{1},j_{2}\leq \sigma$, $0\leq s_{1},s_{2}\leq q-1$ such that $x^{(i_{1}p+j_{1}q+s_{1}q)}=x^{(i_{2}p+j_{2}q+s_{2}q)}$, then we have $q\mid i_{1}-i_{2}$. Hence $i_{1}=i_{2}$. Fix $0\leq i\leq \rho$, the coefficient of $x^{ip+jq}$  in the polynomial$\sum_{i=0}^{\rho}\sum_{j=0}^{\sigma}\sum_{s=0}^{q-1}x^{ip+jq+sq}$ is at most $\sigma+1$ where $0\leq j\leq q+\sigma-1$. So $$\textrm{H}(\sum_{i=0}^{\rho}\sum_{j=0}^{\sigma}\sum_{s=0}^{q-1}x^{ip+jq+sq})\leq\sigma+1.$$
Similarly,
$\textrm{H}(\sum_{i=0}^{q-\rho-2}\sum_{j=0}^{p-\sigma-2}\sum_{s=0}^{q-1}x^{ip+jq+sq+1})\leq p-\sigma-1.$
Since the coefficient of $x^{\sigma q}$ in polynomial $\Phi_{pq}(x)\Phi_{q^{2}}(x)$ is $\sigma+1$ and the coefficient of $x^{(p-(\sigma+2))q+1}$ in polynomial $\Phi_{pq}(x)\Phi_{q^{2}}(x)$ is $p-(\sigma+1)$, we have $$\textrm{H}(\Phi_{pq}(x)\Phi_{q^{2}}(x))= {\rm max}\{(\sigma+1),\ (p-(\sigma+1))\}.$$
Therefore
$\textrm{B}(pq^{3})={\rm max}\{(\sigma+1)p,\ (p-(\sigma+1))p\}.$

 {\bf Case 2.} $p>q.$ For any $h(x)\mid x^{pq^3}-1$, there exist $g(x)\mid\Phi_{q^3}(x)\Phi_{pq^{3}}(x)$ and $f(x)|(x^{pq^2}-1)$ such that $h(x)=f(x)g(x)$. For any
 $f(x)|(x^{pq^2}-1)$, by Lemma 2.2, we have
$$\textrm{H}(\Phi_{q^3}(x)\Phi_{pq^{3}}(x)f(x))=\textrm{H}(\Phi_{q}(x^{pq^2})f(x))\leq {\rm min}\{p,\ q^2\}.$$
It suffices to consider polynomials like $\Phi_{q^{3}}(x)f_{1}(x)f_{2}(x)$ and $\Phi_{pq^{3}}(x)f_{1}(x)f_{2}(x)$ where $f_{1}(x)|\Phi_{p}(x)\Phi_{pq}(x)\Phi_{pq^2}(x)$ and $f_{2}(x)|(x^{q^{2}}-1)$.

First we consider the case $p>q^3$. From Lemmas 2.1 and 2.2, we have $$\textrm{H}(\Phi_{q^{3}}(x)f_{1}(x)f_{2}(x))\leq \textrm{T}(\Phi_{q^{3}}(x))\textrm{H}(f_{1}(x)f_{2}(x))\leq q^{3}.$$
$$\textrm{H}(\Phi_{pq^{3}}(x)f_{1}(x)f_{2}(x))\leq \textrm{T}(f_{2}(x))\textrm{H}(\Phi_{pq^{3}}(x)f_{1}(x))\leq q^{3}.$$
Since $\textrm{B}(pq^3)\geq \textrm{min}\{p,q^3\}$, we have $\textrm{B}(pq^3)=q^3$.

Next we consider the case $q<p<q^3$. There are a few difficult cases which we may consider separately.

(1) $\textrm{H}(\Phi_{q^{3}}(x)\Phi_{pq}(x)f_{2}(x))\leq p$.

(2) $\textrm{H}(\Phi_{pq^{3}}(x)\Phi_{pq}(x)f_{2}(x))\leq 2q$.

(3) $\textrm{H}(\Phi_{pq^{2}}(x)\Phi_{p}(x)f_{2}(x))\leq {\rm max}\{p, 2q\}$.

We prove the first and leave the rest to reader. If $f_{2}(x)=\Phi_{1}(x)\Phi_{q^{2}}(x)$. We have $\Phi_{1}(x)\Phi_{q^{2}}(x)\Phi_{q^{3}}(x)=\sum\limits_{j=0}^{q^2-1}x^{qi+1}-\sum\limits_{i=0}^{q^2-1}x^{qi}$. Now we will show that $$\textrm{H}(\Phi_{1}(x)\Phi_{pq}(x)\Phi_{q^{2}}(x)\Phi_{q^{3}}(x))\leq p.$$ This proof depends on Proposition 2.3. Let $0\leq k \leq p-1$. There do not exist coefficients $a_{n}$, $a_{m}$ of $\Phi_{pq}(x)$ with $n\equiv m \  (\mbox{mod}\  q)$
such that $a_{n}a_{m}=-1$. Suppose the coefficients $a_{n}$ with $n\equiv k \ (\mbox{mod}\  q)$ are all nonnegative. Proposition 2.3 implies that there are $\sigma+1$ positive coefficients in this set. Consider the set of coefficients $a_{l}$ where $l\equiv k-1\  (\mbox{mod}\  q)$.
These $a_{l}$ are either $0$ or $-1$. There are $p-(\sigma+1)$ negative coefficients in this set. Therefore
$\textrm{H}(\Phi_{1}(x)\Phi_{q^{2}}(x)\Phi_{q^{3}}(x)\Phi_{pq}(x))\leq p$. Hence $\textrm{H}(\Phi_{1}(x)\Phi_{pq}(x)\Phi_{q^{2}}(x)\Phi_{q^{3}}(x))\leq p.$ Similarly, if $f_{2}(x)\neq\Phi_{1}(x)\Phi_{q^{2}}(x)$, we have $\textrm{H}(\Phi_{q^{3}}(x)\Phi_{pq}(x)f_{2}(x))\leq p$.

 We will argue that the other divisors of $x^{pq^{3}}-1$ have height at most $\textrm{max}\{p, q^2\}$. We will do this by repeatedly applying lemmas and propositions we have already proven. We present the rest of the proof in the following chart(see Table 1).



\begin{tabular}{|c|c|c|}

\multicolumn{3}{c}{TABLE 1}\\
\multicolumn{3}{c}{}\\
\hline
$f(x)$&$\textrm{H}(\Phi_{pq^{3}}(x)f(x))\leq $&$\textrm{H}(\Phi_{q^{3}}(x)f(x))\leq $\\
\hline
    $f_{2}(x)$&    {$1$}&     $1$\\
\hline
    $\Phi_{p}(x)f_{2}(x)$&    {$p$}&     $p$\\
\hline
    $\Phi_{pq}(x)f_{2}(x)$&     $2q$&    $p$\\
\hline
    $\Phi_{pq^2}(x)f_{2}(x)$&   $p$&     $p$\\
\hline
    $\Phi_{p}(x)\Phi_{pq}(x)f_{2}(x)$&    $p$&     $p$\\
\hline
    $\Phi_{p}(x)\Phi_{pq^2}(x)f_{2}(x)$&     $p$&    max\{$p$,\ $q^2$\}\\
\hline
    $\Phi_{pq}(x)\Phi_{pq^2}(x)f_{2}(x)$&     $2q$&    max\{$p$,\ $2q$\}\\
\hline
    $\Phi_{p}(x)\Phi_{pq}(x)\Phi_{pq^2}(x)f_{2}(x)$&     $p$&    $p$\\
\hline
\end{tabular}

\vskip 2mm

From Table 1, we have

\begin{eqnarray}
\textrm{B}(pq^3)\leq
\begin{cases}
 p, &\textrm{if}\ q^{3}>p>q^2, \cr q^2, & \textrm{if}\ q^{2}>p>q.
\end{cases}
\end{eqnarray}
If $q^{3}>p>q^2$, by Lemma 2.7, we have $\textrm{B}(pq^3)\geq p$ and
then $\textrm{B}(pq^3)=p$. So we only need to consider the case $q^{2}>p>q$. If $h(x)\neq \Phi_{q^{3}}(x)\Phi_{pq^2}(x)\Phi_{p}(x)f_{2}(x)$, we have $\textrm{H}(h(x))\leq {\rm max}\{p,\  2q\}$ follows from Table 1. It is easy to show that
 \begin{eqnarray}
\textrm{H}(\Phi_{p}(x)\Phi_{pq^2}(x)\Phi_{q^{3}}(x)f_{2}(x))\leq
\begin{cases}
q^2, &\textrm{if}\ f_{2}(x)=\Phi_{q}(x),\cr 2q, &\textrm{if}\ f_{2}(x)=\Phi_{1}(x), \Phi_{1}(x)\Phi_{q}(x), \cr p, & \textrm{otherwise}.
\end{cases}
\end{eqnarray}
 So we have proved that if $h(x)\neq \Phi_{q}(x)\Phi_{p}(x)\Phi_{pq^2}(x)\Phi_{q^{3}}(x)$, $\textrm{H}(h(x))\leq {\rm max}\{p,\  2q\}$. Therefore, if $q^2>p>2q$, we have $\textrm{B}(pq^3)= {\rm max}\{p,\  \textrm{H}(\Phi_{p}(x)\Phi_{q}(x)\Phi_{pq^2}(x)\Phi_{q^{3}}(x))\}$. Now we will prove that, if $2q>p>q$, $\textrm{H}(\Phi_{p}(x)\Phi_{q}(x)\Phi_{pq^2}(x)\Phi_{q^{3}}(x))\geq 2q$. We consider two cases.

 \textbf{Case a.}   $q<{\rm min}\{\sigma+1, p-(\sigma+1)\}$. Since the the coefficient of $x^{\sigma q^2+2q-p}$ in the polynomial $\Phi_{p}(x)\Phi_{q}(x)\Phi_{pq^2}(x)\Phi_{q^{3}}(x)$ is $(2q-p+1)q$. However $(2q-p+1)q\geq 2q$, we have $\textrm{H}(\Phi_{p}(x)\Phi_{q}(x)\Phi_{pq^2}(x)\Phi_{q^{3}}(x))\geq 2q.$

 \textbf{Case b.}   $q\geq{\rm min}\{\sigma+1, p-(\sigma+1)\}$. Without loss of generality, we assume that $\sigma+1<p-(\sigma+1).$  Consider the coefficient of $x^{\sigma q^2+q-1}$ in the polynomial $\Phi_{p}(x)\Phi_{q}(x)\Phi_{pq^2}(x)\Phi_{q^{3}}(x)$. If there exist integers $0\leq m\leq p+q-2$, $0\leq n\leq (q-1)(q+p-1)$ such that $x^{m+nq}=x^{\sigma q^2+q-1}$, then we have $m=q-1$ or $m=2q-1$.
If $m=q-1$, we have $n=\sigma q$ and then the coefficient of $x^{q-1}$ of polynomial $\Phi_{p}(x)\Phi_{q}(x)$ is $q$, the coefficient of $x^{\sigma q^2}$ of the polynomial $\Phi_{pq^2}(x)\Phi_{q^3}(x)$ is $\sigma+1$. If $m=2q-1$, we have $n=\sigma q-1$ and then the coefficient of $x^{2q-1}$ in the polynomial $\Phi_{p}(x)\Phi_{q}(x)$ is $p-q$, the coefficient of $x^{\sigma q^2-q}$ in the polynomial $\Phi_{pq^2}(x)\Phi_{q^3}(x)$ is not less than $1-\sigma$.
Therefore the coefficient of $x^{\sigma q^2+q-1}$ in the polynomial $\Phi_{p}(x)\Phi_{q}(x)\Phi_{pq^2}(x)\Phi_{q^{3}}(x)$ is not less than $((\sigma+1)q-(p-q)(\sigma-1))$. However $((\sigma+1)q-(p-q)(\sigma-1))\geq 2q.$ Hence
$\textrm{H}(\Phi_{p}(x)\Phi_{q}(x)\Phi_{pq^2}(x)\Phi_{q^{3}}(x))\geq 2q.$
\qed

 {\bf Example 1.} $\textrm{B}(3\cdot5^3)=6$, $\textrm{B}(5\cdot3^3)=8=\textrm{H}(\Phi_{5}(x)\Phi_{3}(x)\Phi_{5\cdot3^2}(x)\Phi_{3^{3}}(x))$,
$\textrm{B}(7\cdot3^3)=7$, $\textrm{B}(11\cdot3^3)=11$, $\textrm{B}(13\cdot3^3)=13$, $\textrm{B}(29\cdot3^3)=\textrm{B}(31\cdot3^3)=27$.

The following corollary follow directly from Theorems 1.5 and 1.6.
\begin{cor}Let $p<q$ be odd primes and integer $b>2$.

\emph{(1)} If $p>3$, then $\emph{B}(pq^{b})\geq 3p$.

\emph{(2)}  If $p=3$, $\emph{B}(pq^{b})=2p$ if and only if $b=3$ or $4$.
\end{cor}

\begin{cor} Let $p<q<r$ be primes and $q\equiv \pm r\ (\mbox{mod}\  p)$. Then $$\emph{B}(pq^{3})=\emph{B}(pr^{3}).$$\end{cor}
\begin{proof} Assume that $\rho_{1} p+\sigma_{1} q=(p-1)(q-1)$ and $\rho_{2} p+\sigma_{2} r=(p-1)(r-1)$ where $0\leq \sigma_{1},\sigma_{2}< p-1$ and
$0\leq \rho_{1}<q-1$, $0\leq \rho_{2}<r-1$. Then we have $(\rho_{1}-\rho_{2})p+(\sigma_{1} q-\sigma_{2} r)=(p-1)(q-r)$. If $q\equiv r\ (\mbox{mod}\  p)$, we have $(\sigma_{1} q-\sigma_{2} q)\equiv 0\ (\mbox{mod}\  p)$. Therefore $\sigma_{1}=\sigma_{2}$. If $q\equiv -r\ (\mbox{mod}\  p)$, we have $(\sigma_{1} q+\sigma_{2} q)\equiv -2q\ (\mbox{mod}\  p)$. Therefore $\sigma_{1}+\sigma_{2}+2=p$. By Theorem 1.6, we have $\textrm{B}(pq^{3})=\textrm{B}(pr^{3})$.
\end{proof} \qed

\begin{cor} Let $p<q$ be primes. Then $\emph{B}(pq^{3})=p(p-1)$ if and only if $q\equiv \pm 1\ (\mbox{mod}\  p)$.\end{cor}
\begin{proof} By Theorem 1.6, $\textrm{B}(pq^{3})=p(p-1)$ if and only if $\sigma+1=p-1$ or $\sigma=0$.
If $\sigma+1=p-1$, then $q\equiv -1\ (\mbox{mod}\  p)$.
If $\sigma=0$, we have $\rho p=(p-1)(q-1)$ and then $q\equiv 1\ (\mbox{mod}\  p)$. Conversely, if $q\equiv \pm 1\ (\mbox{mod}\  p)$, it is easy to verify that $\textrm{B}(pq^{3})=p(p-1)$.
\end{proof} \qed
\vskip 2mm
\section{\bf Formulas for $\textrm{B}(pq^4)$ and $\textrm{B}(pq^5)$}

\textbf{Proof of Theorem 1.7.}  Note that by Theorem 1.4, it suffices to prove that $\textrm{H}_{4}={\rm max}\{\textrm{H}(g_{0}(x)),  \textrm{H}(\Phi_{pq}(x)\Phi_{pq^2}(x)\Phi_{q^3}(x))\}$ and $\textrm{H}_{5}=\textrm{H}^2(g_{0}(x))$. We will first prove that $\textrm{H}_{4}={\rm max}\{\textrm{H}(g_{0}(x)),  \textrm{H}(\Phi_{pq}(x)\Phi_{pq^2}(x)\Phi_{q^3}(x))\}$. For any $h(x)\mid \prod_{i=1}^{3}\Phi_{q^i}(x^p)$, there exist $g(x)|\prod_{i=2}^{3}\Phi_{q^i}(x^p)$ and $f(x)|\Phi_{q}(x^p)$ such that $h(x)=f(x)g(x)$, then $\textrm{H}(h(x))=\textrm{H}(f(x)g(x))$.
If $f(x)=1$ or $\Phi_{q}(x)$ or $\Phi_{q}(x^p)$, we have $$\textrm{H}(f(x)g(x))=\textrm{H}(g(x))\leq \textrm{H}(g_{0}(x)).$$
If $f(x)=\Phi_{pq}(x)$, by Lemma 2.4, we have $\textrm{H}(f(x)g(x))\leq \textrm{H}(g_{0}(x))\textrm{H}(g(x)).$ By Lemma 2.5, if $g(x)\neq \Phi_{pq^2}(x)\Phi_{q^3}(x)$, we have
$\textrm{H}(g(x))=1$. Hence for any $f(x)|\Phi_{q}(x^p)$, if $g(x)\neq \Phi_{pq^2}(x)\Phi_{q^3}(x)$, we have $\textrm{H}(f(x)g(x))\leq \textrm{H}(g_{0}(x)).$
Therefore $\textrm{H}_{4}={\rm max}\{\textrm{H}(\Phi_{pq}(x)\Phi_{q^2}(x)),  \textrm{H}(\Phi_{pq}(x)\Phi_{pq^2}(x)\Phi_{q^3}(x))\}.$

Next we will show that $\textrm{H}_{5}=\textrm{H}^2(g_{0}(x))$. For any $h(x)\mid \prod_{i=1}^{4}\Phi_{q^i}(x^p)$, there exist $f(x)|\prod_{i=1}^{2}\Phi_{q^i}(x^p)$ and $g(x)|\prod_{i=3}^{4}\Phi_{q^i}(x^p)$ such that $h(x)=f(x)g(x)$, then $\textrm{H}(h(x))=\textrm{H}(f(x)g(x))$. We consider four cases.

(1) If $\textrm{H}(f(x))=\textrm{H}(g(x))=\textrm{H}(g_{0}(x))$,
by Lemma 2.4, we have $$\textrm{H}(f(x)g(x))\leq \textrm{H}(\Phi_{pq}(x)\Phi_{q^2}(x))\textrm{H}(\Phi_{q^2}(x)g(x))=\textrm{H}^2(g_{0}(x)).$$
Consider the coefficients of $x^{\sigma q+\sigma q^3}$ and $x^{(p-\sigma-2)q+(p-\sigma-2)q^3}$ of $f(x)g(x)$, we have $\textrm{H}(f(x)g(x))\geq \textrm{H}^2(g_{0}(x)).$ Hence $\textrm{H}(f(x)g(x))=\textrm{H}^2(g_{0}(x))$.

(2) If $\textrm{H}(f(x))=\textrm{H}(g(x))=1$, then $$\textrm{H}(f(x)g(x))\leq p\textrm{H}(f(x))\textrm{H}(g(x))=p\leq\textrm{H}^2(g_{0}(x)).$$

(3) If $\textrm{H}(f(x))=1$ and $\textrm{H}(g(x))=\textrm{H}(g_{0}(x))$, by Lemma 2.2 and Lemma 2.4, it is easy to show that
$\textrm{H}(f(x)g(x))\leq \textrm{H}(g_{0}(x))\textrm{H}(g(x))\leq\textrm{H}^2(g_{0}(x)).$

(4) If $\textrm{H}(f(x))=\textrm{H}(g_{0}(x))$ and $\textrm{H}(g(x))=1$,
then $$\textrm{H}(f(x)g(x))\leq 2\textrm{H}(f(x))\textrm{H}(g(x))=2\textrm{H}(g_{0}(x)).$$

From above four cases, we have $\textrm{H}(f(x)g(x))\leq \textrm{H}^2(g_{0}(x))$.

Therefore we have $\textrm{B}(pq^5)=p\textrm{H}^2(g_{0}(x))$.

\hfill$\Box$\\

{\bf Example 2.} $\textrm{B}(5\cdot7^{4})= 5\textrm{H}(\Phi_{5\cdot7}(x)\Phi_{5\cdot7^{2}}(x)\Phi_{7^{3}}(x))=20>15=\textrm{B}(5\cdot7^{3})$ and $\textrm{B}(7\cdot17^{4})= \textrm{B}(7\cdot17^{3})=35>28=7\textrm{H}(\Phi_{7\cdot 17}(x)\Phi_{7\cdot 17^2}(x)\Phi_{17^3}(x))$.

\textbf{Proof of Theorem 1.8.}  If $b\leq2$, the Theorem is true since $\textrm{B}(pq^{b})=p$. By Corollary 4.2, the Theorem is true if $b=3$ or 5. By Theorem 1.5, $\textrm{B}(3q^{4})=6$, the Theorem is true if $b=4$ and $p=3$. Then it suffices to consider the case $p>3$ and $b=4$. Since $\textrm{B}(pq^{4})={\rm max}\{\emph{p}\textrm{H}(\Phi_{\emph{pq}}(x)\Phi_{\emph{pq}^{2}}(x)\Phi_{\emph{q}^{3}}(x)),\  \textrm{B}(\emph{pq}^{3})\}$, it is sufficient to show that
 $\textrm{H}(\Phi_{pq}(x)\Phi_{pq^{2}}(x)\Phi_{q^{3}}(x))=\textrm{H}(\Phi_{pr}(x)\Phi_{pr^{2}}(x)\Phi_{r^{3}}(x))$.

  Write $\Phi_{pq}(x)\Phi_{pq^2}(x)\Phi_{q^3}(x)=\sum_{i\geq 0}c_{i}x^{i}$ and $\Phi_{\emph{pr}}(x)\Phi_{\emph{pr}^{2}}(x)\Phi_{\emph{r}^{3}}(x)=\sum_{i\geq 0} d_ix^i.$  Set  $$l_{1}={\rm min}\{\emph{l}\mid |\emph{c}_{\emph{l}}|=\textrm{H}(\Phi_{\emph{pq}}(x)\Phi_{\emph{pq}^{2}}(x)\Phi_{\emph{q}^{3}}(x))\}.$$ Now we will show that $l_{1}<(2p-1)q^{2}$. By Lemma 2.6, we have $l_{1}\equiv 0\ (\mbox{mod}\  q^{2})$ or $l_{1}\equiv 1\ (\mbox{mod}\  q^{2})$. Without loss of generality assume $l_{1}=kq^{2}$. Then $c_{l_{1}}=c_{l_{1}-p}+p$. If $l_{1}\geq(2p-1)q^{2}$, there exist$1<k_{1}<q^{2}$ such that $kq^{2}-k_{1}p=k_{2}q^{2}+1$. If $k_{2}\geq p-1$, we have
$c_{k_{2}q^{2}+1}+p=c_{k_{2}q^{2}+1-p}$. By Lemma 2.6, we have $c_{l_{1}}=c_{k_{2}q^{2}+1-p}$ and $l_{1}>k_{2}q^{2}+1-p$. This contradicts the previous paragraph. So $k_{2}< p-1$ and $l_{1}<(2p-1)q^{2}$. Similarly, we have $l_{1}<(2p-1)q^{2}$ if $l_{1}\equiv 1\ (\mbox{mod}\  q^{2})$.
Next we will show that give any coefficient $c_{n}$ with $n<(2p-1)q^{2}$ of $\Phi_{pq}(x)\Phi_{pq^{2}}(x)\Phi_{q^{3}}(x)$, there exist some coefficient $d_{l_{n}}$ of $\Phi_{pr}(x)\Phi_{pr^{2}}(x)\Phi_{r^{3}}(x)$ with $c_{n}=d_{l_{n}}$.

For $n=[\frac{n}{q^{2}}]q^{2}+n_{0}$ with $0\leq n_{0}<q^{2}$, we can take $l_{n}=[\frac{n}{q^{2}}]r^{2}+n_{0}$. The coefficient of $x^{n}$ in the polynomial $(x^{p}-1)\Phi_{pq}(x)\Phi_{pq^{2}}(x)\Phi_{q^{3}}(x)$ equals to the coefficient of $x^{l_{n}}$ in polynomial $(x^{p}-1)\Phi_{pr}(x)\Phi_{pr^{2}}(x)\Phi_{r^{3}}(x)$, so we have $c_{n}-c_{n-p}=d_{l_{n}}-d_{l_{n}-p}$. We will prove the result $c_{n}=d_{l_{n}}$ by induction. If $n\leq p$, then it is clear that $c_{n}=d_{l_{n}}$. We assume that $c_{k}=d_{l_{k}}$ for $k\leq n$ and consider $c_{n+1}$. Let $n+1=[\frac{n+1}{q^{2}}]q^{2}+n_{0}$ where $0\leq n_{0}<q^{2}$, take $l_{n+1}=[\frac{n+1}{q^{2}}]r^{2}+n_{0}$. If $n_{0}\geq p$, we have $c_{n+1-p}=d_{l_{n+1}-p}$. Therefore $c_{n+1}=d_{l_{n+1}}$ because of $c_{n+1}-c_{n+1-p}=d_{l_{n+1}}-d_{l_{n+1}-p}.$ If $n_{0}<p$, by Lemma 2.6, we have $c_{n+1-p}=d_{[\frac{n+1}{q^{2}}]r^{2}+q^{2}+n_{0}-p}=d_{[\frac{n+1}{q^{2}}]r^{2}+r^{2}+n_{0}-p}=d_{l_{n+1}-p}$. Therefore $c_{n+1}=d_{l_{n+1}}$. So we have $\textrm{H}(\Phi_{pq}(x)\Phi_{pq^{2}}(x)\Phi_{q^{3}}(x))\leq\textrm{H}(\Phi_{pr}(x)\Phi_{pr^{2}}(x)\Phi_{r^{3}}(x))$.

Set $$l_{2}={\rm min}\{\emph{l}\mid |\emph{d}_{\emph{l}}|=\textrm{H}(\Phi_{\emph{pr}}(x)\Phi_{\emph{pr}^{2}}(x)\Phi_{\emph{r}^{3}}(x))\}.$$ Then $l_{2}<(2p-1)r^2$ and $l_{2}\equiv 0\ (\mbox{mod}\  q^{2})$ or $l_{2}\equiv 1\ (\mbox{mod}\  q^{2})$. Without loss of generality assume $l_{2}=kr^{2}$. As in the proof of $c_{n}=d_{l_{n}}$, we have $c_{kq^2}=d_{l_{2}}$. Therefore $$\textrm{H}(\Phi_{pq}(x)\Phi_{pq^{2}}(x)\Phi_{q^{3}}(x))\geq\textrm{H}(\Phi_{pr}(x)\Phi_{pr^{2}}(x)\Phi_{r^{3}}(x)).$$ This completes the proof of
$\textrm{H}(\Phi_{pq}(x)\Phi_{pq^{2}}(x)\Phi_{q^{3}}(x))=\textrm{H}(\Phi_{pr}(x)\Phi_{pr^{2}}(x)\Phi_{r^{3}}(x))$.

 \hfill$\Box$\\

This result shows that given any prime $p$, if we want to find the values of $\textrm{B}(pq^{b})$ where $q>p$, it is sufficient to consider $\frac{p-1}{2}$ distinct prime values of $q$. For example, we can easily get $\textrm{B}(5q^{4})=20$ and the following results
\begin{eqnarray}
\textrm{B}(5q^{3})=
\begin{cases}
20, &\textrm{if}\ q\equiv \pm 1\ (\mbox{mod}\  5),\cr 15, &\textrm{if}\ q\equiv \pm 2\ (\mbox{mod}\  5).
\end{cases}
\end{eqnarray}
\begin{eqnarray}
\textrm{B}(7q^{3})=
\begin{cases}
42, &\textrm{if}\ q\equiv \pm 1\ (\mbox{mod}\  7),\cr 28, &\textrm{if}\ q\equiv \pm 2\ (\mbox{mod}\  7),\cr 35, &\textrm{if}\ q\equiv \pm 3\ (\mbox{mod}\  7).
\end{cases}
\end{eqnarray}
\begin{eqnarray}
\textrm{B}(7q^{4})=
\begin{cases}
42, &\textrm{if}\ q\equiv \pm 1\ (\mbox{mod}\  7),\cr 35, &\textrm{otherwise}.
\end{cases}
\end{eqnarray}

 Ryan et al ([8]) have computed $\textrm{B}(n)$ for almost 300000 distinct $n$. All of the form $n=pq^b$ satisfy $\textrm{B}(pq^{b})=\textrm{B}(pr^{b})$
where $p<q<r$ and $q\equiv \pm r\ (\mbox{mod}\  p)$. One is tempted to believe the following

{\bf Conjecture 5.1.} Let $p<q<r$ be distinct odd primes. If $q\equiv \pm r\ (\mbox{mod}\  p)$, then $\textrm{B}(pq^{b})=\textrm{B}(pr^{b})$.

It is clear that the same propositions used in this section will be useful in studying $\textrm{B}(p^aq^b)$, however it is not at all clear what the resulting formula should be, even for $\textrm{B}(pq^4)$ if $p>q$ or $\textrm{B}(p^2q^2)$. The approach involves carefully analyzing many divisors and becomes impractical when the divisors becomes large.

{\noindent\bf Acknowledgments}
  I would like to thank Professor Qin Hourong and Xia Jianguo for their ideals and suggestions for this paper. I would like to thank Ji Qingzhong for several helpful suggestions.
I would also like to thank my friends, Fangyong and Wang Junfeng, who gave me their help and time in listening to me and helping me work out my problems during the difficult course of the paper.

\end{document}